\documentclass[11pt]{article}
\usepackage[utf8]{inputenc}
\usepackage[T1]{fontenc}
\DeclareUnicodeCharacter{0229}{\k{e}}
\usepackage{lmodern}
\usepackage{subfiles}
\usepackage{enumitem}
\setenumerate{topsep=6pt,ref={\normalfont(\roman*)},label={\normalfont(\roman*)}, itemsep=0pt} 

\usepackage{amsfonts}
\usepackage{amsthm}
\usepackage{amsmath}
\usepackage{amssymb}
\usepackage{amscd}
\usepackage{mathrsfs}
\usepackage{mathtools}
\usepackage{bbm}
\usepackage{esint}

\usepackage[margin=2.7cm]{geometry}
\usepackage{setspace}
\usepackage{indentfirst}
\usepackage{graphicx}
\usepackage{graphics}
\usepackage{lscape}
\usepackage{pgf,tikz}
\usepackage{tikz-cd}
\usepackage{color}
\usepackage{pict2e}
\usepackage{epic}
\usepackage{epstopdf}
\usepackage{titlesec, titlefoot}
\titleformat{\section}[block]{\Large\bfseries\filcenter}{\thesection}{1em}{}
\titleformat{\part}[block]{\LARGE\bfseries\filcenter}{Part \thepart.}{0.5em}{}
\usepackage{commath}
\usepackage{float}
\usepackage{caption}
\usepackage{etoolbox}
\usepackage{combelow}

\usepackage[hidelinks,bookmarksdepth=3]{hyperref}
\hypersetup{bookmarksopen=true} 

\graphicspath{{./Pictures/}}
\allowdisplaybreaks

\expandafter\def\expandafter\normalsize\expandafter{%
\normalsize
\setlength\abovedisplayskip{6pt}
\setlength\belowdisplayskip{6pt}
\setlength\abovedisplayshortskip{6pt}
\setlength\belowdisplayshortskip{6pt}
}

\theoremstyle{plain}

\renewcommand*\thesection{\arabic{section}}
\numberwithin{equation}{section} 

\newtheorem{theorem}{Theorem}[section]
\newtheorem{lemma}[theorem]{Lemma}
\newtheorem*{lemma*}{Lemma}

\newtheorem{corollary}[theorem]{Corollary}

\theoremstyle{definition}

\newtheorem{remark}[theorem]{Remark}

\expandafter\let\expandafter\oldproof\csname\string\proof\endcsname
\let\oldendproof\endproof
\renewenvironment{proof}[1][\proofname]{%
\oldproof[\upshape \bfseries #1]%
}{\oldendproof}

\makeatletter
\def\@makechapterhead#1{%
\vspace*{50\p@}%
{\parindent \z@ \raggedright \normalfont
\interlinepenalty\@M
\Huge\bfseries  \thechapter.\quad #1\par\nobreak
\vskip 40\p@
}}
\makeatother

\renewcommand{\Im}{\operatorname{Im}}

\DeclareMathOperator{\rank}{rank}

\let\textcaron\v
\renewcommand{\v}[1]{\ifmmode\check{#1}\else\textcaron{#1}\fi}
\def \bs{\boldsymbol}

\def \R {\mathbb{R}}
\def \Z {\mathbb{Z}}

\def \H{\mathbb{H}}
\def \D{\textup{D}}

\def \T{\mathbb{T}}
\def \e{\varepsilon}
\def \d{\,\textup{d}}

\def \p{\partial}
\def \mc{\mathcal}
\def \mb{\mathbb}

\def \tp{\textup}

\renewenvironment{thebibliography}[1]{
  \begin{oldthebibliography}{#1}
    \setlength{\itemsep}{0.5pt}
    \setlength{\parskip}{0.5pt}
}{
  \end{oldthebibliography}
}

\begin{document}

	\title{\textbf{Morrey's problem in $\R^{2\times 4}$ and $\R^{3\times 3}_\mathrm{sym}$}}
		
	\author{
  { Andrea Agazzi},\  
  { Giuseppe Bruno},\  
  { Andr\'e Guerra},\  
  { Federico Pasqualotto}
}
				
	\date{}
	
	\maketitle
	
	\begin{abstract}
    We find an explicit rank-one convex non-quasiconvex integrand in $\R^{2\times 4}$: to falsify the quasiconvexity inequality, we exhibit a map $\mb T^4\to \R^2$ with 12 non-zero Fourier modes.
    In fact, this map is obtained from a scalar potential, so we also find a rank one convex integrand in $\R^{4\times 4}_\tp{sym}$ which is not quasiconvex. These examples are obtained by transpositions and restrictions  of Grabovsky's example of a rank-one convex, non-quasiconvex integrand in $\R^{8\times 2}.$ We also modify \v Sver\'ak's example to construct a rank-one convex non-quasiconvex integrand in $\R^{3\times 3}_\tp{sym}$.
    \end{abstract}

\section{Introduction}

Let $F\colon \R^{m\times n}\to \R$ be locally Lipschitz and let $\mb T^n$ be the periodic torus. We say that $F$ is:
\begin{enumerate}
\item\label{it:qcvx} \textit{quasiconvex} if, for all $A\in \R^{m\times n}$ and all $\varphi \in C^\infty(\T^n,\R^m)$, we have
$$\int_{\T^n} [F(A+ \D \varphi) -F(A)] \d x\geq 0;$$
\item\label{it:rcvx} \textit{rank-one convex} if, for all $A,X\in \R^{m\times n}$ with $\rank(X)=1$, $t\mapsto F(A+t X)$ is convex.
\end{enumerate} 
It is well-known that \ref{it:qcvx}$\implies$\ref{it:rcvx}, see e.g.\ \cite{Dacorogna2007}. 
A main open problem, first posed by Morrey \cite{Morrey1952,Morrey1966}, is whether the converse holds. This is trivial if either $m=1$ or $n=1$, so we assume that $m,n\geq 2$. The importance of quasiconvexity in the Calculus of Variations is that it is essentially equivalent to the \textit{lower semicontinuity} \cite{Dacorogna2007} and to \textit{mean coercivity} \cite{Chen2017} of the associated functional.

 It is now known that rank-one convexity does not imply quasiconvexity if either $m\geq 3$ \cite{Sverak1992a} or $m=2$ and $n\geq N_0$, where $N_0\approx 1000$ \cite{Cassese2026}. Note that the latter work  \cite{Cassese2026} contains many examples of rank-one convex non-quasiconvex homogeneous integrands which furthermore have additional symmetries, including also the case $m\geq 4, n=2$. The case $n=m=2$, however, remains open, and there is some evidence that the answer may be different in this case \cite{Astala2022,Astala2012,Bruno2026,Conti2003,Faraco2008,GuerraCosta2020,Harris2018,Kirchheim2008,Muller1999b,Sebestyen2017,Szekelyhidi2005}.
 
\v Sver\'ak's example in \cite{Sverak1992a} consists of an explicit
rank-one convex integrand (in fact a polynomial of degree $4$), together
with a map with three Fourier modes,
\begin{equation}\label{eq:sverak-map}
\varphi_{\v S}(x_1,x_2)
:=\sin x_1\,e_1+\sin x_2\,e_2
+\sin(x_1+x_2)\,e_3,
\end{equation}
where $\{e_1,e_2,e_3\}$ is the canonical basis of $\R^3$. Its gradient
takes values in the linear space
 \begin{equation}
 \label{eq:L}
L:=\left\{
\begin{bmatrix}r&0\\0&s\\t&t\end{bmatrix}:
r,s,t\in\R
\right\},
 \end{equation}
and a simple calculation shows that the integrand
$F_{\v S}\colon L\to \R$ which maps the matrix parametrized by
$(r,s,t)$ to $-rst$ is rank-one convex, yet
\[
\int_{\T^2}
\bigl[F_{\v S}(\D\varphi_{\v S})-F_{\v S}(0)\bigr]
\d x<0.
\]
A somewhat related example had also been previously observed by Tartar \cite{Tartar1979}.
One can then extend $F_{\v S}$ to $\R^{3\times 2}$ while preserving rank-one convexity and non-quasiconvexity.

A further example, with $m=8$ and $n=2$, was constructed by Grabovsky \cite{Grabovsky2018}.
In this note we focus on his example, which is again completely explicit:
$$G\colon \H^2 \times \H^2 \to [0,\infty),\qquad G(\boldsymbol\xi,\boldsymbol\eta):=\sqrt{\|\boldsymbol\xi\|_{\H^2}^2 \|\boldsymbol\eta\|_{\H^2}^2-|\langle \boldsymbol\xi,\boldsymbol\eta\rangle_{\H^2}|^2}.$$
Here and throughout, bold symbols denote quaternions or quaternionic vectors, and $\langle\boldsymbol\xi,\boldsymbol{\eta}\rangle_{\H^2}:=\boldsymbol\xi_1 \boldsymbol{\bar \eta_1} + \boldsymbol{\xi_2}\boldsymbol{\bar \eta_2}$ as usual.
Identifying $\H^{2\times 2}\cong \R^{8\times 2}$ in the natural way, we find an integrand $G\colon \R^{8\times 2}\to [0,\infty)$ which we denote by the same symbol. Grabovsky proved:
\begin{theorem}[\cite{Grabovsky2018}]\label{thm:grab}
The integrand $G$ is 2-homogeneous, $\tp{SO}(2)$-invariant, rank-one convex, but non-quasiconvex at $((\boldsymbol{1},\boldsymbol 0),(\boldsymbol{0},\boldsymbol{1}))$, which corresponds to the identity $\boldsymbol I_2\in \H^{2\times 2}$.
\end{theorem}

Our first observation is that, somewhat surprisingly, Grabovsky's
example can be seen as an extension of \v Sver\'ak's example. Consider the real-linear map
\[
S\colon\R^3\longrightarrow\H^2,
\qquad
S(r,s,t):=(-s\mathbf j-t\mathbf k,r\mathbf i+t\mathbf k).
\]
Let $\boldsymbol e_1=(1,0)$ and $\boldsymbol e_2=(0,1)$ be the
standard basis of $\H^2$. For
$A=(A_1\mid A_2)\in\R^{3\times2}$, with $A_1,A_2\in\R^3$, define $F\colon \R^{3\times 2}\to \R$ by
\begin{equation}\label{eq:def-F32}
F(A):=G(\boldsymbol e_1+SA_1,\boldsymbol e_2+SA_2).
\end{equation}
Then, on the space $L$ in \eqref{eq:L},
\[
F\left(\begin{bmatrix}r&0\\0&s\\t&t\end{bmatrix}\right)^2
=1-4rst+r^2s^2+r^2t^2+s^2t^2.
\]
Thus the leading order term behavior of $F$ is very similar to that of $F_{\v S}$, and
in fact $\varphi_{\v S}$ also disproves the quasiconvexity
of $F$, see Lemma \ref{lemma:grab}. We remark that Grabovsky's proof of non-quasiconvexity of $G$ is very indirect, and an explicit map $\varphi$ falsifying the quasiconvexity inequality is never exhibited.

Now let $\Im \H:=\{w_1\mathbf i+w_2\mathbf j+w_3\mathbf k:(w_1,w_2,w_3)\in \R^3\}$ denote the set of purely imaginary quaternions.
We consider the restriction of Grabovsky's integrand to
$\R^{6\times2}$ and its transposition:

\begin{theorem}\label{thm:main}
Let
$G_{\mathrm{im}}
:=G\big|_{(\Im\H)^2\times(\Im\H)^2}
\colon\R^{6\times2}\to[0,\infty)$
and define its transpose by
\[
G_{\mathrm{im}}^T\colon\R^{2\times6}\to[0,\infty),
\qquad A\mapsto G(A^T).
\]
Both $G_{\mathrm{im}}$ and $G_{\mathrm{im}}^T$ are rank-one convex but not quasiconvex.
\end{theorem}

We emphasize that while rank-one convexity is invariant under
transposition, quasiconvexity in general is not. In fact, the example of
\cite{Sverak1992a} is not invariant under transposition
\cite{Muller2000}, see also \cite[Remark 4.4]{Cassese2026} for another example. The key point in the proof of Theorem \ref{thm:main}, and the second main contribution of this paper, is to identify a suitable map $\varphi\colon \mb T^6\to \R^2$ contradicting quasiconvexity of $G_{\mathrm{im}}^T$. The map we identify consists of 12 modes, which appears to be the smallest possible number for this example, see Remark \ref{rem:intuition}. It turns out that the specific map $\varphi$ we find depends only on 4 (out of 6) variables. Restricting $G_{\rm{im}}^T$ to an appropriate affine subspace determined by the range of $\D \varphi$, we find:

\begin{corollary}\label{thm:F24}
There is an explicit rank-one convex integrand $F_{2,4}\colon \R^{2\times 4}\to [0,\infty)$, obtained by restricting $G_{\rm{im}}^T$ to an appropriate affine subspace, which is not quasiconvex.
\end{corollary}

We note that the integrand in Corollary \ref{thm:F24} is no longer 2-homogeneous.
Leveraging the fact that the map $\varphi$ in Theorem \ref{thm:main} and  Corollary \ref{thm:F24} in fact is obtained from a scalar potential we also get:

\begin{corollary}\label{cor:scalar}
There exists an explicit integrand $H\colon \R^{4\times4}_{\rm sym}\to[0,\infty)$  which is convex along symmetric rank-one directions, in the sense that $t\mapsto  H(A+t\,a\otimes a)$ is convex for every $A\in\R^{4\times4}_{\rm sym}$ and $a\in\R^4$, but which is not Hessian-quasiconvex at $0$: there exists $\psi\in C^\infty(\T^4)$ such that \[ \fint_{\T^4} H(\D^2\psi)\d x< H(0). \] 
\end{corollary} 

A variant of \v Sver\'ak's construction, now with a map with 4 modes, gives the same conclusion in dimension three:

\begin{theorem}\label{thm:sym3}
There exists a quartic polynomial
$F\colon \R^{3\times3}_{\rm sym}\to\R$
which is convex along every symmetric rank-one direction, but is not
Hessian-quasiconvex at 0: there is
$\Psi\in C^\infty(\T^3)$ with
\[
\fint_{\T^3}F(\D^2\Psi)\,\d x<F(0).
\]
\end{theorem}

For further details on quasi- and rank-one convexity on symmetric matrices see \cite{Maso2003,Faraco2003,Sverak1992}. We emphasize that it has been known for a long time \cite[Example 3.5]{Ball1981} that, for vector-valued maps, the analogue of Corollary \ref{cor:scalar} and Theorem \ref{thm:sym3} hold, so the key novelty here is that the maps $\psi,\Psi$ are \textit{scalar}. To the best of our knowledge, prior to this work there was no example showing the difference between quasiconvexity and rank-one convexity in $\R^{n\times n}_\tp{sym}$, regardless of the value of $n$. The case of $n=2$ is particularly outstanding.

Since $G$ and $F_{\v S}$ are so closely related, and the same map $\varphi_{\v S}$ is used to disprove their quasiconvexity, the result of \cite{Muller2000} strongly suggests that these examples cannot lead to a rank-one convex, non-quasiconvex integrand in $\R^{2\times 3}$. Nonetheless, we believe there should be such an example.

\subsection*{Acknowledgements}

 The authors acknowledge Gabriele Cassese for comments on a preliminary version of this manuscript. AA acknowledges the partial support of the Swiss National Science Foundation project grant 2000-1-243054. GB acknowledges support from Thang Luong and Garrett Bingham. AG acknowledges the support of the Royal Society through a Newton International Fellowship. He would also like to thank Y.\ Grabovsky for many interesting discussions about his example. FP acknowledges support from the Bergman Fellowship of the American Mathematical Society.

\subsection*{AI disclosure}
AI, in particular ChatGPT-5.5, ChatGPT-5.6 sol Pro and Gemini / Aletheia2,  played a crucial role in the research underlying this paper. The authors prompted the models to use Grabovsky's example to find further counter-examples in lower dimensions. This process generated, largely autonomously, the core constructions presented in Theorem \ref{thm:main}, concerning $G_{\tp{im}}^T$, and Corollary \ref{thm:F24}. The authors then further used AI to digest these constructions, and found the relationship between Grabovsky's and \v Sver\'ak's examples. This paper has been written completely and checked entirely by the authors, who take full responsibility for its contents.

\section{Grabovsky's and \v Sver\'aks's examples}

In this section, we make a connection between $F_{\v S}$ and $G$. Here and throughout, dashed integrals denote averaged integrals.

\begin{lemma}\label{lemma:grab}
The integrand $F\colon\R^{3\times2}\to[0,\infty)$ defined in
\eqref{eq:def-F32} is rank-one convex and is not quasiconvex at $0$.
More precisely, the three-mode map $\varphi_{\v S}$ defined in
\eqref{eq:sverak-map}
satisfies
\[
\fint_{\T^2}F(\D\varphi_{\v S})\,\d x
=\fint_{\T^2}G\bigl(I_2+\D(S\varphi_{\v S})\bigr)\,\d x
\leq\frac{\sqrt3}{2}<1=F(0)=G(I_2).
\]
\end{lemma}

\begin{proof}
Let $H=a\otimes\zeta\in\R^{3\times2}$ have rank one, so its columns are $H_1=\zeta_1 a, H_2=\zeta_2 a$. Since
\[
(SH_1,SH_2)=(Sa)\otimes\zeta\in\R^{8\times2},
\]
the affine map
\[
A\mapsto
(\boldsymbol e_1+SA_1,\boldsymbol e_2+SA_2)
\]
maps rank-one lines to rank-one lines. Rank-one convexity of $F$
therefore follows from that of $G$. 

For
\begin{equation}
    \label{eq:defA}
    A=\begin{bmatrix}r&0\\0&s\\t&t\end{bmatrix}\in L,
\end{equation}
the two quaternionic vectors in \eqref{eq:def-F32} are
\[
\boldsymbol\xi=(1-t\mathbf k,r\mathbf i+t\mathbf k),
\quad 
\boldsymbol\eta=(-s\mathbf j-t\mathbf k,1+t\mathbf k), 
\]
Consequently, and using 
$\boldsymbol{k}\boldsymbol j=-\boldsymbol{i},\boldsymbol{i}\boldsymbol{k}=-\boldsymbol{j},\boldsymbol{k}^2=-\boldsymbol{1}$,
\begin{align*}
|\boldsymbol\xi|^2&=1+r^2+2t^2,
\qquad |\boldsymbol\eta|^2&=1+s^2+2t^2,\qquad
\langle\boldsymbol\xi,\boldsymbol\eta\rangle_{\H^2}
=2t^2+(r+st)\mathbf i+(s+rt)\mathbf j+2t\mathbf k.
\end{align*}
Substitution in the definition of $G$ gives
\begin{equation}\label{eq:F32-on-L}
F\left(\begin{bmatrix}r&0\\0&s\\t&t\end{bmatrix}\right)^2
=1-4rst+r^2s^2+r^2t^2+s^2t^2.
\end{equation}

We have 
\[
\D\varphi_{\v S}
=\begin{bmatrix}a&0\\0&b\\c&c\end{bmatrix},
\qquad
a:=\cos x_1,\qquad b:=\cos x_2,\qquad
c:=\cos(x_1+x_2),
\]
and hence, by \eqref{eq:F32-on-L},
\[
F(\D\varphi_{\v S})^2
=1-4abc+a^2b^2+a^2c^2+b^2c^2.
\]
We then have
\[
\fint_{\T^2}abc\,\d x=\frac14,
\qquad
\fint_{\T^2}a^2b^2\,\d x
=\fint_{\T^2}a^2c^2\,\d x
=\fint_{\T^2}b^2c^2\,\d x
=\frac14.
\]
It follows that
\[
\fint_{\T^2}F(\D\varphi_{\v S})^2\,\d x=\frac34.
\]
Jensen's inequality therefore yields
\[
\fint_{\T^2}F(\D\varphi_{\v S})\,\d x
\leq
\left(\fint_{\T^2}F(\D\varphi_{\v S})^2\,\d x\right)^{1/2}
=\frac{\sqrt3}{2}<1=F(0).
\]
Finally, linearity of $S$ gives
$
F(\D\varphi_{\v S})
=G\bigl(\boldsymbol I_2+\D(S\varphi_{\v S})\bigr),
$
which proves the remaining claims.
\end{proof}

\section{Proof of Theorem \ref{thm:main} and Corollary \ref{thm:F24}}

\begin{proof}[Proof of Theorem \ref{thm:main}]
It is clear that, being the restriction and transposition of a rank-one convex integrand, both $G_{\rm im}$ and $G_{\rm im}^T$ are rank-one convex.

Let us first show the non-quasiconvexity of $G_{\mathrm{im}}$.
Set
\[ R:\mathbb H^2\to\mathbb H^2, \qquad R(\boldsymbol q_1,\boldsymbol q_2):=(\boldsymbol q_1\boldsymbol i,\boldsymbol q_2\boldsymbol j). \]
Note that $R$ is an isometry: \[ \langle R\boldsymbol\xi,R\boldsymbol\eta\rangle_{\mathbb H^2} = \boldsymbol\xi_1\boldsymbol i\,\overline{\boldsymbol{\eta}_1\boldsymbol i} +\boldsymbol{\xi_2j}\,\overline{\boldsymbol \eta_2 \boldsymbol j} = \langle\boldsymbol \xi,\boldsymbol \eta\rangle_{\mathbb H^2}\quad \implies \quad  G(R\boldsymbol \xi,R\boldsymbol \eta)=G(\boldsymbol \xi,\boldsymbol \eta). \] Using the identities
$\boldsymbol{ji}=-\boldsymbol{k},\boldsymbol{ki}=\boldsymbol j,\boldsymbol{ij}=\boldsymbol k,\boldsymbol{kj}=-\boldsymbol i$,
we compute
\[ RS(r,s,t)=(s\boldsymbol k-t\boldsymbol j,r\boldsymbol k-t\boldsymbol i)\in(\operatorname{Im}\mathbb H)^2. \] Let $\varphi_{\check S}$ be the map from Lemma \ref{lemma:grab} and set $\psi:=RS\varphi_{\check S}:\mathbb T^2\to(\operatorname{Im}\mathbb H)^2.$ Then  
\[ \fint_{\mathbb T^2} G_{\rm im}\bigl(R\boldsymbol I_2+\D\psi\bigr)\d x = \fint_{\mathbb T^2} G\bigl(\boldsymbol I_2+\D(S\varphi_{\check S})\bigr)\d x <G(\boldsymbol I_2) =G_{\rm im}(R\boldsymbol I_2), \] 
where $R$ also acts column-wise on matrices.
Hence $G_{\rm im}$ is not quasiconvex.

To prove the non-quasiconvexity of $G_{\rm im}^T$, let
$A\in\R^{2\times6}$ and write
\[
A=\begin{bmatrix}u&v\\s&t\end{bmatrix},
\qquad u,v,s,t\in\R^3.
\]
For $r=(r_1,r_2,r_3)\in\R^3$, we use the bold symbol
$
\boldsymbol r:=r_1\mathbf i+r_2\mathbf j+r_3\mathbf k\in\Im\H.
$
The identity
\[
\boldsymbol a\boldsymbol b=-a\cdot b+ \boldsymbol{a\times b},
\qquad a,b\in\R^3,
\]
then shows that
$$\langle (\boldsymbol u, \boldsymbol v), (\boldsymbol s,\boldsymbol t)\rangle_{\H^2} = \boldsymbol{u \bar s} + \boldsymbol{v \bar t}=-\boldsymbol{u s}-\boldsymbol{v t}= u\cdot s + v\cdot t -(\boldsymbol u \times \bs s+ \bs v \times \boldsymbol t)$$
and so, using $|a\times b|^2=|a|^2|b|^2-(a\cdot b)^2$ and $(a\times b)\cdot(c\times d)=(a\cdot c)(b\cdot d)-(a\cdot d)(b\cdot c)$, 
\begin{align}
\label{eq:GimT-as-Q}
\begin{split}
    G_{\mathrm{im}}^T(A)^2& =(|\bs u|^2 + |\bs v|^2)(|\bs s|^2+|\bs t|^2)-(u\cdot s +v\cdot t)^2-|\bs u \times \bs s + \bs v \times \bs t|^2\\
    & = |\bs u|^2|\bs t|^2 + |\bs s|^2 |\bs v|^2-2 (u\cdot s)(v\cdot t)-2(u\cdot v)(s\cdot t)+ 2(u\cdot t)(s\cdot v).
     \end{split}
\end{align}

We use coordinates
$(y_0,y_1,y_2,z_0,z_1,z_2)\in \T^6$. We first set
$$\varphi_{\tau}:= \bigl(\partial_{z_1}\psi_{\tau}, -\partial_{y_1}\psi_{\tau}\bigr)$$
for a scalar potential $\psi_\tau\colon \mb T^4\to \R$, depending only on $(y_1,y_2,z_1,z_2)$, to be found. We regard $\varphi_{\tau}\colon \T^6\to \R^2$ as a map which is independent of $y_0$ and $z_0$, so we have
$$
\D \varphi_\tau = 
\begin{bmatrix}
0 & \partial_{y_1z_1}\psi_\tau & \partial_{y_2z_1}\psi_\tau & 0 & \partial_{z_1z_1}\psi_\tau & \partial_{z_1z_2}\psi_\tau \\[2mm] 0 & -\partial_{y_1y_1}\psi_\tau & -\partial_{y_1y_2}\psi_\tau & 0 & -\partial_{y_1z_1}\psi_\tau & -\partial_{y_1z_2}\psi_\tau.
\end{bmatrix}
$$
Crucially, we require $\psi_\tau$ to satisfy
\begin{equation}
\label{eq:keyid}
    \p_{y_2 z_1} \psi_\tau = \p_{y_1 z_2} \psi_\tau.
\end{equation}
This condition, together with our choice of $\varphi_\tau$, is imposed so that the second order in $\varepsilon$ term below vanishes after integration ($S_\tau$ below is a null Lagrangian, cf.\ \eqref{eq:nullLag}). With the matrix
$$
M=\begin{bmatrix}
    -2 & 0 & 0 & 0 & 0 & 0\\
    0 & 0 & 0 &-2 & 0 & 0
\end{bmatrix},
\qquad u=t=-2 e_1, v=s=0,
$$
using \eqref{eq:GimT-as-Q} and $G_{\rm{im}}^T\geq 0$ we have $G_{\mathrm{im}}^T(M)=4$. By \eqref{eq:keyid}, the corresponding vector $u,v,s,t$ for the matrix $M+\e \D \varphi_\tau$ are 
\begin{align*}
    u &= (-2 ,\e \p_{y_1 z_1}\psi_\tau, \e \p_{y_1 z_2}\psi_\tau), 
    & v& = (0,\e \p_{z_1 z_1}\psi_\tau, \e \p_{z_1 z_2}\psi_\tau),\\
s& =(0,-\e \p_{y_1y_1}\psi_\tau,-\e\p_{y_1 y_2}\psi_{\tau}),&
    t& =(-2,-\e \p_{y_1 z_1}\psi_\tau, -\e \p_{y_1 z_2}\psi_\tau).
\end{align*}
After substituting these expressions into \eqref{eq:GimT-as-Q}, a long but elementary calculation gives
\begin{equation}\label{eq:GimT-exact-density}
G_{\mathrm{im}}^T(M+\e\D\varphi_\tau)^2
=(4-\e^2S_\tau )^2+\e^4 E_\tau,
\end{equation}
where
\begin{align}
     S_\tau & := \det \D^2_{(y_1,z_1)}\psi_\tau + \det\!\left(\D_{(y_1,z_1)}\D_{(y_2,z_2)}\psi_\tau\right), \nonumber \\ 
     \begin{split}
     \label{eq:defE}
         E_\tau & :=  \left( \partial_{y_1y_1}\psi_\tau\,\partial_{z_1z_2}\psi_\tau -\partial_{z_1z_1}\psi_\tau\,\partial_{y_1y_2}\psi_\tau \right)^2\\ 
     &\qquad +4\left( \partial_{y_1y_1}\psi_\tau\,\partial_{y_1z_2}\psi_\tau -\partial_{y_1z_1}\psi_\tau\,\partial_{y_1y_2}\psi_\tau \right) \left( \partial_{z_1z_1}\psi_\tau\,\partial_{y_1z_2}\psi_\tau -\partial_{y_1z_1}\psi_\tau\,\partial_{z_1z_2}\psi_\tau \right)\\
     &=: C_1^2 + 4 C_2 C_3.
     \end{split}
\end{align}
Note that, since the determinant is a null Lagrangian, we have
\begin{equation}
\label{eq:nullLag}
\fint_{\T^6}S_\tau \,\d x=0.
\end{equation}

So far we have computed \eqref{eq:GimT-exact-density} for a general map $\varphi_\tau$, assuming only \eqref{eq:keyid}, and  we now choose $\psi_\tau$. For $\tau\in\R$ to be chosen, we set
\begin{align}
\label{eq:choicepsi}
\begin{split}
g_\tau(\theta_1,\theta_2)
&:=\cos\theta_1+\cos\theta_2+\cos(\theta_1+\theta_2)
+\tau\cos(\theta_1-\theta_2),\\
\psi_\tau(y_1,y_2,z_1,z_2) &:=2g_\tau(y_1-y_2,z_1-z_2) -g_\tau(2y_1,2z_1) +2g_\tau(y_1+y_2,z_1+z_2).
\end{split}
\end{align}
One can easily check that \eqref{eq:keyid} holds.
We now write the Fourier expansion of  $\psi_\tau$ in its 12 modes:
$$\psi_\tau(x)=\sum_{k\in \Z^4} c_k \cos(k\cdot x),\qquad 
\p_{x_i x_j} \psi_\tau(x) = -\sum_{k\in \Z^4}c_k k_i k_j \cos(k\cdot x),
\qquad x:=(y_1,y_2,z_1,z_2),$$
where the only non-zero coefficients are
\[ \begin{array}{c|c} k & c_k\\ \hline (1,-1,0,0) & 2\\ (0,0,1,-1) & 2\\ (1,-1,1,-1) & 2\\ (1,-1,-1,1) & 2\tau\end{array} 
\qquad 
\begin{array}{c|c} k & c_k\\ \hline 
(2,0,0,0) & -1\\ (0,0,2,0) & -1 \\
(2,0,2,0) & -1\\ (2,0,-2,0) & -\tau
\end{array} 
\qquad 
\begin{array}{c|c} k & c_k\\ \hline (1,1,0,0) & 2\\ (0,0,1,1) & 2\\ (1,1,1,1) & 2\\ (1,1,-1,-1) & 2\tau \end{array} 
\]
Thus, for the constants in \eqref{eq:defE}, we have
\begin{align*} C_1 &= \sum_{k,\ell\in \Z^4} c_kc_\ell \left( k_1^2\ell_3\ell_4 -k_3^2\ell_1\ell_2 \right) \cos(k\cdot x)\cos(\ell\cdot x),\\ 
C_2 &= \sum_{k,\ell\in \Z^4} c_kc_\ell \left( k_1^2\ell_1\ell_4 -k_1k_3\ell_1\ell_2 \right) \cos(k\cdot x)\cos(\ell\cdot x),\\ 
C_3 &= \sum_{k,\ell\in \Z^4} c_kc_\ell \left( k_3^2\ell_1\ell_4 -k_1k_3\ell_3\ell_4 \right) \cos(k\cdot x)\cos(\ell\cdot x),\end{align*}
and using  $\cos(k\cdot x)\cos(\ell\cdot x) = \frac12\cos((k+\ell)\cdot x) + \frac12\cos((k-\ell)\cdot x),$
we obtain the Fourier expansions of $C_1,C_2,C_3$.
With the help of a computer algebra software, one can then compute
\begin{equation}\label{eq:GimT-averages}
\fint_{\T^6} E_\tau \,\d x=96\tau(17\tau-2).
\end{equation}
A Mathematica notebook symbolically verifying this calculation, as well as the calculation of \eqref{eq:GimT-exact-density}, is available in  \cite{morrey_notebook_github}. The calculation of the average of $E_\tau$ is very difficult to perform by hand, as one has to determine 84 scalar coefficients.

Since $S_\tau, E_\tau$ are bounded,
\eqref{eq:GimT-exact-density} gives
\begin{equation}
    \label{eq:expG}
    G_{\rm im}^T(M+\e \D \varphi_\tau)=4-\e^2 S_\tau+ \frac{\e^4}{8}E_\tau+O(\e^6)
\end{equation}
and so
\eqref{eq:nullLag}-\eqref{eq:GimT-averages} give, after setting $\tau=\frac{1}{12}$,
\[
\lim_{\e\to0}\frac{1}{\e^4}
\left(
\fint_{\T^6}G_{\mathrm{im}}^T(M+\e\D\varphi_{1/12})\,\d x
-G_{\mathrm{im}}^T(M)
\right)
=-\frac7{12}.
\]
Consequently, for every sufficiently small nonzero $\e$,
\begin{equation}\label{eq:GimT-violation}
\fint_{\T^6}G_{\mathrm{im}}^T(M+\e\D\varphi_{1/12})\d x
<G_{\mathrm{im}}^T(M)=4.
\end{equation}
This proves the failure of quasiconvexity of $G_{\mathrm{im}}^T$ and
completes the proof.
\end{proof}

\begin{remark}\label{rem:intuition}
We make here some remarks about the construction and, in particular, the key choice in \eqref{eq:choicepsi}.
First, from \eqref{eq:GimT-as-Q}, it is easy to see that 
$$G_{\mathrm{im}}^T(A)^2=2\,\lvert \operatorname{sym}O(A)\rvert^2-\operatorname{tr}O(A)^2, \qquad O(A):=u\otimes t-s\otimes v,$$
where $\operatorname{sym} O(A)$ is the symmetric part of $O(A)$.
Since \(O\) is quadratic, we can write
$$O(M+\e\D\varphi)=O_0+\e O_1+\e^2O_2,\qquad O_0:=O(M),$$
so that $G_{\mathrm{im}}^T(M+\e\D\varphi)^2$ contains terms up to order
$\e^4$. The construction can be understood as eliminating
these terms successively.
First, we note that taking 
$$O_1=O'(M)[\D\varphi]=-2
\begin{bmatrix}
0&\varphi^2_{z_1}&\varphi^2_{z_2}\\
\varphi^1_{y_1}&0&0\\
\varphi^1_{y_2}&0&0
\end{bmatrix},
$$
to be skew-symmetric cancels its contribution to $G_{\mathrm{im}}^T(M+\e\D\varphi)^2$ in a pointwise sense. 
This condition is automatically satisfied by taking
$$\varphi=(\psi_{z_1},-\psi_{y_1})$$
and imposing \eqref{eq:keyid}, and allows to neglect the term $O_1$ altogether in our analysis. This leads to an expansion in even powers of $\epsilon$ of $G_{\mathrm{im}}^T(M+\e\D\varphi)^2$ and, in turn, of $G_{\mathrm{im}}^T(M+\e\D\varphi)$ as in \eqref{eq:GimT-exact-density} and \eqref{eq:expG} respectively.
Proceeding with the $\e^2$ order in the perturbative expansion \eqref{eq:expG}, we note that our choice of $\varphi$ above also cancels (upon integration) the coefficient $S_\tau$ by making it a null Lagrangian, see \eqref{eq:nullLag}.

We now consider the remaining $\e^4$ term $E_\tau$ defined in \eqref{eq:defE}. For a Fourier mode $e^{i  k\cdot x}$, where again $x=(y_1,y_2,z_1,z_2),$ if we write $k=(k'_1,k''_1,k'_2,k''_2)$ with $k',k''\in \Z^2$, the condition \eqref{eq:keyid} enabling the cancellation of the lower order terms is exactly
$$k\in \mathcal C :=\{(k',k'')\in \Z^4: \det(k',k'')=0\}.$$
Given a frequency $m\in \Z^2$, there is a particularly simple choice of frequencies, lying in $\mc C$, and which satisfy an identity similar to the one in \v Sver\'ak's map $\varphi_{\v S}$, namely
$$(m,-m),(m,m), (2m,0)\in \mc C, \qquad (m,-m) + (m,m)=(2m,0).$$
This choice corresponds to lifting a scalar profile $g\colon \T^2 \to \R$ to a profile on $\T^4$, through
$$\psi(y_1,y_2,z_1,z_2)=2 g(y_1-y_2,z_1-z_2) +2 g(y_1+y_2,z_1+z_2)- g(2y_1,2z_1)$$
where the numerical coefficients in from of each wave are so that, upon differentiation, all waves have the same amplitude. 

Following this reasoning, and starting with the map $\varphi_{\v S}$ from Lemma \ref{lemma:grab}, we would obtain the functions $g_0,\psi_0$ in \eqref{eq:choicepsi}, the latter having  9 Fourier modes. However, as seen from \eqref{eq:GimT-averages}-\eqref{eq:expG}, the choice $\tau=0$ leads to a map $\varphi_0$ which is in the kernel of the 4th-order variation, hence it may not (and in fact will not) falsify the quasiconvexity inequality. Our construction \eqref{eq:choicepsi} perturbs this choice of $g_0$ in a nondegenerate direction of the fourth order term $E_\tau$ (adding 3 more modes, for a total of 12, to the perturbation) while respecting the constraints canceling the lower terms in the expansion, leading to a negative overall variation of the integral.
\end{remark}

\begin{proof}[Proof of Corollary \ref{thm:F24}]
Define the linear embedding 
\[ \iota:\R^{2\times4}\to\R^{2\times6}, \qquad \iota \begin{bmatrix} a_1&a_2&a_3&a_4\\ b_1&b_2&b_3&b_4 \end{bmatrix} := \begin{bmatrix} 0&a_1&a_2&0&a_3&a_4\\ 0&b_1&b_2&0&b_3&b_4 \end{bmatrix}, \] 
and let $M$ be as in the previous proof. Set \[ F_{2,4}(A):=G_{\rm im}^T\bigl(M+\iota(A)\bigr), \qquad A\in\R^{2\times4}. \] Then $F_{2,4}$ is rank-one convex but not quasiconvex at $0$. Indeed, $\iota$ preserves rank-one directions: if $A=a\otimes\xi\in\R^{2\times4}$ has rank one, then $\iota(A) = a\otimes(0,\xi_1,\xi_2,0,\xi_3,\xi_4),$ so rank-one convexity follows from that of $G_{\rm im}^T$. For the failure of quasiconvexity we let $$\widetilde \varphi_{1/12}(y_1,y_2,z_1,z_2):=\varphi_{1/12}(y_0,y_1,y_2,z_0,z_1,z_2)$$ be the restriction of the map from the previous proof to $\T^4$, so that \[ D\varphi_{1/12}=\iota(D\widetilde\varphi_{1/12}). \] Hence, by \eqref{eq:GimT-violation}, for every sufficiently small nonzero $\varepsilon$, \[ \begin{aligned} \fint_{\T^4} F_{2,4}\bigl(\varepsilon D\widehat\varphi_{1/12}\bigr)\,dx = \fint_{\T^6} G_{\rm im}^T\bigl(M+\varepsilon D\varphi_{1/12}\bigr)\,dx < G_{\rm im}^T(M) = F_{2,4}(0). \end{aligned} \] Thus $F_{2,4}$ is not quasiconvex at $0$.
\end{proof}

\begin{proof}[Proof of Corollary \ref{cor:scalar}] Set 
\[
H(A):=F_{2,4}(TA), \qquad  T:= \begin{pmatrix} 0&0&1&0\\ -1&0&0&0 \end{pmatrix}. 
\] 
For every $a\in\R^4$, $T(a\otimes a)=(Ta)\otimes a,$ which has rank at most one. Hence, by the rank-one convexity of $F_{2,4}$, the function $ t\to  H(A+t\,a\otimes a) = F_{2,4}(TA+t(Ta)\otimes a)$ is convex. Now let $\psi_{1/12}$ be the scalar function constructed in the proof of Theorem~1.2. Thus 
\[ 
\varphi_{1/12} = \bigl(\partial_{z_1}\psi_{1/12}, -\partial_{y_1}\psi_{1/12}\bigr) = T\nabla\psi_{1/12}, 
\quad \implies \quad 
D\varphi_{1/12}=TD^2\psi_{1/12}.
\] 
Therefore, by the proof of Corollary \ref{thm:F24}, for every sufficiently small nonzero $\varepsilon$, 
\[ \fint_{\T^4}  H\left(\varepsilon \D^2\psi_{1/12}\right)\,dx = \fint_{\T^4} F_{2,4}\left(\varepsilon \D\varphi_{1/12}\right)\,dx < F_{2,4}(0) = H(0). \] 
The claim follows.
\end{proof}

\section{A symmetric \texorpdfstring{$3\times3$}{3 x 3} example}

\begin{proof}[Proof of Theorem \ref{thm:sym3}]
Let
\[
\xi_1=e_1,\qquad \xi_2=e_2,\qquad \xi_3=e_3,\qquad
\xi_4=e_1+e_2+e_3,
\]
and set
\[
A_j:=\xi_j\otimes\xi_j,
\qquad
L:=\operatorname{span}\{A_1,A_2,A_3,A_4\}.
\]
The matrices $A_1,\ldots,A_4$ are linearly independent. Thus every
$X\in L$ has a unique representation
\[
X=X(r):=\sum_{j=1}^4r_jA_j
=
\begin{pmatrix}
r_1+r_4&r_4&r_4\\
r_4&r_2+r_4&r_4\\
r_4&r_4&r_3+r_4
\end{pmatrix}.
\]
We claim that
\begin{equation}\label{eq:sym3-rank-one-L}
L\cap\{X\in\R^{3\times3}_{\rm sym}:\rank X\leq1\}
=\bigcup_{j=1}^4\R A_j.
\end{equation}
Indeed, if $\rank X(r)\leq1$, three of its non-principal minors give
\[
r_1r_4=r_2r_4=r_3r_4=0.
\]
If $r_4\neq0$, then $r_1=r_2=r_3=0$. If $r_4=0$, the matrix $X(r)$
is diagonal, and at most one of $r_1,r_2,r_3$ is nonzero. This proves
\eqref{eq:sym3-rank-one-L}.

Define $f\colon L\to\R$ by
\[
f(X(r)):=-r_1r_2r_3r_4.
\]
By \eqref{eq:sym3-rank-one-L}, every symmetric rank-one direction in
$L$ is a multiple of some $A_j$. Since $f$ is affine in each
coordinate, it is convex along every such direction.

Now set
\[
\Psi(x_1,x_2,x_3)
:=\cos x_1+\cos x_2+\cos x_3+\cos(x_1+x_2+x_3).
\]
Then
\[
\D^2\Psi
=-\cos x_1\,A_1-\cos x_2\,A_2-\cos x_3\,A_3
-\cos(x_1+x_2+x_3)\,A_4,
\]
so $\D^2\Psi(x)\in L$ for every $x\in\T^3$. Writing
\[
r_1=-\cos x_1,\quad r_2=-\cos x_2,\quad r_3=-\cos x_3,
\quad r_4=-\cos(x_1+x_2+x_3),
\]
Fourier orthogonality gives
\[
\fint_{\T^3}r_1r_2r_3r_4\,\d x=\frac18, \qquad 
\fint_{\T^3}f(\D^2\Psi)\,\d x=-\frac18<0=f(0).
\]

The extension from $L$ to $\R^{3\times3}_{\rm sym}$ follows exactly as
in \cite{Sverak1992a}, using the perturbed functional
\[
F_{\alpha,\lambda,\mu}(Z)
:=f(PZ)+\alpha|PZ|^4
+\lambda|PZ|^2|QZ|^2+\mu|QZ|^4.
\]
Here $P$ is the orthogonal projection onto $L$, $Q:=I-P$, and
$\alpha,\lambda,\mu>0$ are chosen as in \cite{Sverak1992a}.
This concludes the proof.
\end{proof}

{\small
\bibliographystyle{abbrv-andre}
\bibliography{library}
}
\vspace{2em}
\noindent
{\small
\textsc{Andrea Agazzi} \\
Department of Mathematics and Statistics, University of Bern \\
Alpeneggstrasse 22, 3012 Bern, CH \\
\texttt{andrea.agazzi@unibe.ch}

\vspace{1em}
\noindent
\textsc{Giuseppe Bruno} \\
Department of Mathematics and Statistics, University of Bern \\
Alpeneggstrasse 22, 3012 Bern, CH \\
\textit{and} \\
Google DeepMind \\
\texttt{giuseppe.bruno@unibe.ch}

\vspace{1em}
\noindent
\textsc{Andr\'e Guerra} \\
Department of Pure Mathematics and Mathematical Statistics, University of Cambridge \\
Wilberforce Rd, Cambridge CB3 0WB, UK \\
\texttt{adblg2@cam.ac.uk}

\vspace{1em}
\noindent
\textsc{Federico Pasqualotto} \\
Department of Mathematics, University of California San Diego \\
La Jolla, CA 92093, USA \\
\textit{and} \\
Google DeepMind \\
\texttt{fpasqualotto@ucsd.edu}
}

\end{document}